\documentclass[a4j]{article}
\usepackage{comment}
\usepackage{mathrsfs,bm}
\usepackage{amssymb}
\usepackage{amsthm}
\usepackage{amsmath}
\usepackage{amsfonts}
\usepackage{cases}
\usepackage{braket}

\theoremstyle{plain} 
\newtheorem{thm}{Theorem}[section]
\newtheorem{cor}[thm]{Corollary}
\newtheorem{lem}[thm]{Lemma}
\newtheorem{prop}[thm]{Proposition}

\newtheorem{defn}{Definition}[section]
\newtheorem{rem}{Remark}[section]

\newcommand{\head}[0]{\mathfrak{h}}
\newcommand{\tail}[0]{\mathfrak{t}}
\newcommand{\inv}[0]{^{-1} }
\newcommand{\arc}[0]{\mathfrak{A}}

\newcommand{\circtheta}[0]{\mbox{\rm circ}_\theta}
\newcommand{\re}[0]{\mbox{\rm Re }}
\newcommand{\im}[0]{\mbox{\rm Im }}

\title{A family of quantum walks on a finite graph corresponding to the generalized weighted zeta function}

\author{Ayaka Ishikawa \\
 Faculty of Engineering, Yokohama National University, \\
 Hodogaya, Yokohama 240-8501, Japan}

\date{}

\begin{document}

\maketitle

\begin{abstract}
	This paper gives the quantum walks determined by graph zeta functions.
	The result enables us to obtain the characteristic polynomial of the transition matrix of the quantum walk, and it determines the behavior of the quantum walk.
	We treat finite graphs allowing multi-edges and multi-loops.

\end{abstract}


\section{Introduction}


This paper considers the relationship between a graph zeta function and a quantum walk.
A quantum walk is a quantization of a random walk.
The same as random walks, mixing time, hitting time, etc., are studied.
In addition, there is characteristic behavior of quantum walk, including localization and periodicity.
It is determined by the spectrum of the transition matrix of a quantum walk \cite{Higuchi2007, konno2017, segawa2011}.
For the Grover walk \cite{grover1996fast}, Konno and Sato \cite{konnosato12} indicated the relation between the Grover walk and the Sato zeta function, and they gave the spectrum of the transition matrix by the Sato zeta function.
The Sato zeta function \cite{sato07} is a generalized graph zeta function of the Ihara zeta function \cite{ihara66}.

The essential point of the theorem lies in the fact that the ``edge matrix'' $M_{\rm S}$ of the Sato zeta function is a generalization of the transition matrix of the Grover walk $U_{\rm Gr}$.
We call the theorem Konno-Sato's theorem.
If we impose certain conditions into the Sato zeta function, then $^t\!M_{\rm S} = U_{\rm Gr}$ holds.
The Sato zeta function is given by the inverse of the reciprocal characteristic polynomial of the edge matrix, called the Hashimoto expression.
Thus, we can obtain the characteristic polynomial of $U_{\rm Gr}$ by the inverse of the reciprocal Sato zeta function.

The problem treated in this paper is to identify the conditions under which $^t\!M_{\rm S}=U$ for the transition matrix $U$ of a quantum walk.
Our theorem shows the condition for the existence of a quantum walk with $^t\!M_{\rm S}$ as the transition matrix.
The result gives the family of quantum walks whose behavior is identified by the Sato zeta function.

Identifying the relationship between the other graph zeta function and quantum walks is also a problem.
We consider the same problem for the generalized weighted zeta function \cite{IIMSS21, morita20}.
The zeta function is a generalization of the graph zeta functions, including the Sato zeta function.
Thus, the family identified by the generalized weighted zeta function contains the Grover walk.

Throughout this paper, we use the following symbols.
Let the set of the positive integers by $\mathbb{Z}_>$.
The spectrum of a matrix $M$ is ${\rm Spec}(M)$.
For a vector $\Psi$, $||\Psi||$ is the $L^2$-norm of $\Psi$.

\section{Preliminary}\label{sec:pre}

\subsection{Graphs}

Let $G=(V, E)$ be a graph, and the edge set $E$ be a multiset.
The graph is {\it finite} if both $V$ and $E$ are finite.
For a vertex $v\in V$, let $\deg(v) := \#\{ \{v,w\} \in E | w\in V \}$.
It is called the {\it degree} of $v$.
If there is at most one edge between every two vertices,
then the graph is called {\it simple}.
Let $\arc$ be a multiset of ordered pairs of vertices,
and the element is called {\it arcs}.
A {\it digraph} $\Delta$ is a pair $(V,\arc)$.
The digraph is {\it finite} if both $V$ and $\arc$ are finite.
For an arc $a=(v,w)$, 
$v$ and $w$ are called the {\it tail} and {\it head} of $a$.
We denote by $\tail(a):=v$ and $\head(a):=w$.
For two vertices $v,w\in V$,
we define
\begin{align*}
	&\arc_{vw} := \{ a\in\arc \ | \ \tail(a)=v, \head(a)=w \}, \\
	&\arc_{v*} := \{ a\in\arc | \tail(a)=v\}, \\
	&\arc_{*w} := \{ a\in\arc \ | \ \head(a)=w\}, \\
	&\arc(v,w) := \arc_{v,w}\cup \arc_{wv}.
\end{align*}
For a graph $G=(V,E)$,
let $\arc(G):=\{ a_e(v,w), \overline{a}_e=(w,v) | e=\{ w,v \}\in E \}$.
The digraph $\Delta(G)=(V,\arc(G))$ is called the {\it symmetric digraph} for $G$.
Note that $\deg(v)=|\arc_{v*}|$ holds for $v\in V$.

For a digraph $\Delta=(V,\arc)$,
a sequence of arcs $p=(a_i)_{i=1}^{k}$ is a {\it path} if it satisfies $\head(a_i) = \tail(a_{i+1})$ for each $i=1,2,\ldots, k-1$.
The number $k$ is called the {\it length} of $p$.
A {\it backtracking} is a path $(a,a')$ satisfying $\head(a')=\tail(a)$. 
The path $p$ is {\it closed} if $\head(a_k)=\tail(a_1)$ holds.
Let $X_k$ be the set of closed paths with length $k$ on $\Delta$.
For $p\in X_k$ and $n\in\\mathbb{Z}_>$, $p^n$ denotes a closed path with length $kn$ obtained by joining $n$ paths $p$.
If $p$ does not have a backtracking and closed path $q$ s.t. $p=q^m$, then it is called a {\it prime} closed path.

Let $\sigma$ be a map onto $X_k$ s.t. for $p=(a_i)_{i=1}^{k}$, $\sigma(p)=(a_2,a_3,\ldots ,a_k,a_1)$.

For $p=(a_i)_{i=1}^k, p'=(a_i')_{i=1}^k \in X_k$,
we define a relation $\sim$ between $p$ and $p'$ if there exists a positive integer $n$ s.t. $\sigma^n(p)=p'$.
The relation is an equivalence relation.
The quotient set $X_k/\sim$ is denoted by $\mathcal{P}_k$.
We call an equivalence class in $\mathcal{P}_k$ a {\it cycle}, and $[p]$ denotes the cycle including $p$.

\subsection{Quantum walks on graphs}
Let $G=(V,E)$ be a finite simple graph, and $\Delta(G)=(V,\arc(G))$ the symmetric digraph for $G$.
Let a $\mathcal{H}$$\mathbb{C}$-liner space $\mathcal{H}$ defined as follows:
$$
	\mathcal{H}:=\{ \Psi : \arc(G) \to \mathbb{C} \ | \ ||\Psi(a)||^2 < \infty \},
$$
and we assume an inner product on $\mathcal{H}$ as the Euclidian inner product.
Then, $\mathcal{H}$ is the Hilbert space.
Let $\delta_a$ for $a\in\arc(G)$ be a function satisfying
$$
	\delta_a(a')=\begin{cases}
			1 & \mbox{if} \ a=a', \\
			0 & \mbox{otherwise},
		\end{cases}
$$
and we regard the set $\{ \delta_a | a\in\arc(G) \}$ as a standard basis on $\mathcal{H}$.
We assume that a function $w : \arc(G)\to \mathbb{C}$ satisfies 
$$
	\sum_{a\in\arc_{*v}} |w(a)|^2 =1.
$$
The {\it coin matrix} $C$ is the following unitary matrix:
$$
	(C\Psi)(a)=\sum_{a'\in\arc_{*\head(a)}} w(a') \Psi(a').
$$
The unitary matrix $U:=SC$ is called the {\it transition matrix}.
The {\it quantum walk} is a process defined by a transition matrix and an initial state on $\mathcal{H}$.
For an initial state $\Psi_0\in\mathcal{H}$ with $||\Psi_0||^2=1$, the {\it state at time $n$} $\Psi_n$ is $U^n\Psi_0$.
The probability of observing on $v\in V$ at time $n$ is given by $\sum_{a\in\arc_{*v}} ||(U^n\Psi)(a)||^2$.

Let $C_{\rm Gr}$ be the following coin matrix:
$$
	(C_{\rm Gr}\Psi)(a)=\sum_{a'\in\arc_{*\head(a)}} \left( \frac{2}{\deg(\head(a))}-\delta_a(a') \right) \Psi(a')
$$
for $\Psi\in\mathcal{H}$.
The transition matrix $U_{\rm Gr}:=SC_{\rm Gr}$ is called the {\it Grover transition matrix},
and the quantum walk decided by $U_{\rm Gr}$ is called {\it Grover walk} on $G$ \cite{grover1996fast}.
Note that $(a,a')$-element of $U_{\rm Gr}$ is given as follows:
$$
	\frac{2}{\deg(\tail(a))}\delta_{\head(a')\tail(a)} - \delta_{a'\overline{a}}.
$$

\begin{rem}
	For instance, the spectrum of $U$ is convenient for knowing the periodicity of a quantum walk.
	A transition matrix $U$ is {\it periodic} if there exists $k\in \mathbb{Z}_>$ satisfying $U^k = I$.
	The minimum value of such $k$ is the {\it period} of $U$.
	Note that if $U$ is periodic, then $U^{nk}\Psi_0=\Psi_0$ holds for $\forall n\in\mathbb{Z}_>$ and any initial state $\Psi_0$.
	It is known that if $\forall \mu\in{\rm Spec}(U)$ is a primitive root of $n(\mu)\in \mathbb{Z}_>$, then the period equals ${\rm LCM}( n(\mu) )_{\mu\in{\rm Spec}(U)}$.
	Thus, the spectrum of $U$ allows the periodicity to be determined without simulation.
\end{rem}

Let $T:=(T_{uv})_{u,v\in V}$ be a matrix defined as follows:
	$$
		T_{uv} =\begin{cases}
			\frac{1}{\deg(u)} & \mbox{ if } \  \{ u,v \}\in E,\\
			0 & \mbox{ otherwise}.
		\end{cases}
	$$
Konno-Sato's theorem \cite{konnosato12} gives the characteristic polynomial of $U_{\rm Gr}$ by the eigenvalues of $T$.
\begin{thm}\label{thm:Grover}
	Let $G=(V,E)$ be a finite simple connected graph with $n$ vertices and $m$ edges.
	The characteristic polynomial of $U_{\rm Gr}$ is given as follows: 
	\begin{align}
		\det(\lambda I- U_{\rm Gr}) &= (\lambda^2-1)^{m-n} \det((\lambda^2+1)I-2\lambda T) \label{eq:Grover}\\
		&= (\lambda^2-1)^{m-n} \prod_{\mu \in {\rm Spec}(T)} ((\lambda^2+1)-2\mu\lambda).\nonumber
	\end{align}
	From the above, we have
	$$
	{\rm Spec}(U_{\rm Gr}) = \{ -1, 1\}^{m-n} \sqcup \{ \lambda \ | \ \lambda^2-2\mu\lambda +1 =0, \ \mu\in {\rm Spec}(T) \}.
$$
\end{thm}
By the transformation from the Hashimoto expression to the Ihara expression of the Sato zeta function, (\ref{eq:Grover}) is given.
The following section will mention a graph zeta function and its expressions.


\subsection{Graph zeta function}
Let $\Delta=(V,\arc)$ be a finite digraph.
We define a map $\theta : \arc\times \arc \to \mathbb{C}$.
Let $M_\theta:=(\theta(a,a'))_{a,a'\in\arc}$.
For a closed path $C=(c_i)_{i=1}^k\in X_k$, 
let ${\rm circ}_\theta (C)$ denote the circular product 
$
	\theta(c_1,c_2) \theta(c_2,c_3) \ldots \theta(c_k,c_1).
$
Note that ${\rm circ}_\theta(C)={\rm circ}_\theta(C')$ holds if $C\sim C'$.
Let $N_k({\rm circ}_\theta) := \sum_{C\in X_k} {\rm circ}_\theta(C)$.

\begin{defn}
	A graph zeta function for $\Delta$ is the following formal power series:
	\begin{align*}
		Z_\Delta(t;\theta) := \exp \left(\sum_{k\ge 1}\frac{N_k({\rm circ}_\theta)}{k}t^k \right).
	\end{align*}
\end{defn}
The map $\theta$ is called the {\it weight} of $Z_\Delta(t;\theta)$, and that expression is called the {\it exponential expression} \cite{morita20}.
Let
\begin{align*}
	E_\Delta (t;\theta) := \prod_{[C]\in\mathcal{P}} \frac{1}{1-{\rm circ}_\theta (C)t^{|C|}}, \qquad 
	H_\Delta (t;\theta) := \frac{1}{\det(I-tM_\theta)}.
\end{align*}
The expressions are called the {\it Euler expression} and {\it Hashimoto expression}, respectively (cf. \cite{morita20}).
\begin{prop}
	If $\theta : \arc\times\arc \to \mathbb{C}$ satisfies the condition 
	$$
		\theta(a,a') \neq 0 \Rightarrow \head(a)=\tail(a'),
	$$
	then $Z_\Delta(t;\theta)= E_\Delta (t;\theta)=H_\Delta (t;\theta)$ holds.
\end{prop}
\begin{proof}
	See \cite{morita20}.
\end{proof}
The above condition for $\theta$ is called the {\it adjacency condition}.

Let $G=(V,E)$ be a finite graph allowing multi-edges and multi-loops, and $\Delta(G)=(V,\arc(G))$ the symmetric digraph for $G$.
For maps $\tau, \upsilon : \arc \to \mathbb{C}$,
we define $\theta:=\theta^{\rm GW}$ by
$$
	\theta^{\rm GW} (a,a') := \tau(a')\delta_{\head(a)\tail(a')} - \upsilon(a')\delta_{\overline{a} a'}.
$$
Then $Z_{\Delta(G)}(t;\theta)$ is called the {\it generalized weighted zeta function}.

\begin{rem}\label{rem:weight}
{\rm
	For example, the generalized weighted zeta function includes the following graph zeta functions:
	\begin{itemize}
	\item Ihara zeta function $(\theta^{\rm I} := \theta|_{\tau(a)=\upsilon(a)=1})$ \cite{ihara66},
	\item Bartholdi zeta function $(\theta^{\rm B} := \theta|_{\upsilon(a)=(q-1)\tau(a)})$  \cite{bartholdi99},
	\item Mizuno-Sato zeta function $(\theta^{\rm MS} := \theta|_{\tau(a)=\upsilon(a)})$ \cite{mizuno2004weighted},
	\item Sato zeta function $(\theta^{\rm S} :=\theta|_{\upsilon(a) = 1})$ \cite{sato07}. 
	\end{itemize}}
\end{rem}

If $\theta(a,a')\neq 0$ for $a,a'\in\arc(G)$, then $\delta_{\head(a) \tail(a')} =1$ holds at least.
Thus, the weight satisfies the adjacency condition, and we see $Z_{\Delta(G)}(t;\theta)= E_{\Delta(G)} (t;\theta)=H_{\Delta(G)} (t;\theta)$.

Let $A_\theta:=(A_{uv})_{u,v\in V}, D_\theta:=(D_{uv})_{u,v\in V}$ be defined by
	\begin{align*}
		A_{uv} := \sum_{a\in \arc_{uv}} \frac{\tau(a)}{1-t^2\upsilon(a)\upsilon(\overline{a})}, \quad
		D_{uv} := \delta_{uv}\sum_{a\in \arc_{u*}} \frac{\tau(a)\upsilon(\overline{a})}{1-t^2\upsilon(a_e)\upsilon(\overline{a_e})}.
	\end{align*}
	We call these matrices the {\it weighted adjacency matrix} and {\it weighted degree matrix}, respectively.
	Note that $D_\theta$ is a diagonal matrix.
\begin{prop}\label{prop:ihara}
	Let $\arc(G)=\{ a_e,\overline{a_e} \ | \ e\in E \}$.
	The generalized weighted zeta function $Z_{\Delta(G)}(t;\theta^{\rm GW})$ is given by
	\begin{align*}
		\prod_{e\in E} (1-t^2\upsilon(a_e)\upsilon(\overline{a_e})) \det(I-tA_\theta +t^2D_\theta).
	\end{align*}
\end{prop}

Let $G=(V, E)$ be a finite connected graph.
Theorem \ref{thm:Grover} follows from the Ihara expression of the Sato zeta function $Z_{\Delta(G)}(t;\theta^{\rm S})$.
For the weight $\theta^{\rm S}=\theta|_{\upsilon(a) = 1}$ in Remark \ref{rem:weight},
let $\tau(a)=\frac{2}{\deg(\tail(a))}$ for $\forall a\in\arc(G)$.
An $(a,a')$-element of the edge matrix $M_{\rm S}$ of the Sato zeta function is as follows:
$$
	\frac{2}{\deg(\tail(a'))}\delta_{\head(a)\tail(a')}-\delta_{\overline{a}a'}.
$$
We obtain $U_{\rm Gr}= ^t\!M_{\rm S}$, and $H_{\Delta(G)}(t;\theta^{\rm S})$ gives the characteristic polynomial $\det(I-tU_{\rm Gr})$.

Since the weighted adjacency matrix and weighted degree matrix equal $2T$ and $2I$, respectively,
we get Theorem \ref{thm:Grover}.

\section{Main result}\label{main}

\subsection{The quantum walks following from the Sato zeta function}

Let $G=(V, E)$ be a finite graph allowing multi-edges and multi-loops.
We will show the condition for the existence of a quantum walk that has the transition matrix $^t\!M_{\rm S}$.
If $^t\!M_{\rm S}$ is a transition matrix, then the sift matrix and coin matrix are given by $S$ and $S\inv {^t\!M}_{\rm S}$, respectively.
Since a coin matrix is just a unitary matrix, we only need to obtain the condition that $S\inv {^t\!M}_{\rm S}$ is unitary.
The unitary conditions for $S\inv {^t\!M}_{\rm S}$, ${^t\!M}_{\rm S}$ are equivalent since $S$ is unitary.
Thus, we show the unitary condition for $M_{\rm S}$.

\begin{thm}\label{thm:sato}
	The edge matrix $M_{\rm S}=(\theta^{\rm S}(a,a'))_{a,a'\in\arc(G)}$ of the Sato zeta function is unitary if and only if the map $\tau$ satisfies
	$\tau(a)=\tau(a')$ for $\forall u\in V$ and $\forall a\in\arc_{a,a'\in\arc_{u*}}$.
\end{thm}
\begin{proof}
	Let $\arc:=\arc(G)$ and assume that $M_{\rm S}$ is unitary.
	Note that if $a,b\in\arc$ satisfies $\overline{a} = b$, then $\head(a)=\tail(b)$ holds, and we have
	\begin{align*}
		\delta_{\head(a)\tail(b)} - \delta_{\overline{a} b} = \delta_{\head(a)\tail(a')}\left( 1 - \delta_{\overline{a}b} \right).
	\end{align*}
	The $(a,a')$-element of the complex conjugate transpose $M_{\rm S}*$ is given by
	\begin{align*}
		 \overline{\tau(a)}\delta_{\head(a')\tail(a)} - \delta_{\overline{a'} a}.
	\end{align*}

	The $(a,a')$-element of $M_{\rm S}^*M_{\rm S}$ is as follows:
	\begin{align*}
		& \sum_{b\in\arc}
				\left( \overline{\tau(a)}\delta_{\head(b)\tail(a)} - \delta_{\overline{a} b} \right)
				\left( \tau(a')\delta_{\head(b)\tail(a')} - \delta_{\overline{b} a'} \right) \nonumber\\
		& = \delta_{\tail(a)\tail(a')} \sum_{b \in\arc_{*\tail(a)}}
				\left( \overline{\tau(a)} - \delta_{\overline{a} b} \right)
				\left( \tau(a') - \delta_{\overline{b} a'} \right) \nonumber\\
		& = \delta_{\tail(a)\tail(a')} \sum_{b \in\arc_{*\tail(a)}}
				\left( \overline{\tau(a)}\tau(a') - \overline{\tau(a)}\delta_{\overline{b} a'} -\tau(a')\delta_{\overline{a} b} +\delta_{\overline{b} a'}\delta_{\overline{a} b} \right) \nonumber\\
		& = \begin{cases}
			\deg(\tail(a)) |\tau(a)|^2 - 2|\tau(a)| \cos(\arg \tau(a)) + 1 & {\rm if}\quad a=a', \\
			\deg(\tail(a)) \overline{\tau(a)}\tau(a') - \overline{\tau(a)} -\tau(a') &  {\rm if}\quad  a\neq a' \ {\rm and} \  \tail(a)=\tail(a'), \\
			0 & {\rm otherwise}. 
		\end{cases}
	\end{align*}
	We assume that $(M_{\rm S}^*M_{\rm S})_{a,a'}=\delta_{aa'}$.
	For the $(a,a)$-element, we have
	\begin{align*}
		&\deg(\tail(a)) |\tau(a)|^2 - 2|\tau(a)| \cos(\arg \tau(a)) + 1=1 \\
		&\Leftrightarrow |\tau(a)|\left( \deg(\tail(a))|\tau(a)|-2\cos(\arg \tau(a)) \right)=0 \\
		&\Leftrightarrow |\tau(a)| = 0 \quad {\rm or} \quad |\tau(a)|=\frac{2\cos(\arg \tau(a))}{\deg(\tail(a))}.
	\end{align*}
	Since $|\tau(a)|=0$ holds if $\arg \tau(a)=\frac{n\pi}{2}$ for $n\in\mathbb{Z}_>$, the case $|\tau(a)|=\frac{2\cos(\arg \tau(a))}{\deg(\tail(a))}$ contains the case $|\tau(a)| = 0$.
	The $(a,a')$-element satisfying $a\neq a'$ and $\tail(a)=\tail(a')$ will be discussed later.

	The $(a,a')$-element of $M_{\rm S}M_{\rm S}^*$ is as follows:
	\begin{align}
		& \sum_{b\in\arc}
		\left( \tau(b)\delta_{\head(a)\tail(b)} - \delta_{\overline{a} b} \right)
		\left( \overline{\tau(b)}\delta_{\head(a')\tail(b)} - \delta_{\overline{a'} b} \right) \nonumber\\
		&= \delta_{\head(a)\head(a')} \sum_{b\in\arc_{\head(a)*}}
				\left( \tau(b) - \delta_{\overline{a} b} \right)
			\left( \overline{\tau(b)} - \delta_{\overline{a'} b} \right) \nonumber\\
		&=\delta_{\head(a)\head(a')}\sum_{b\in\arc_{\head(a)*}}
				\left( |\tau(b)|^2 - \tau(b)\delta_{\overline{a'} b}
			- \overline{\tau(b)}\delta_{\overline{a} b} + \delta_{\overline{a} b} \delta_{\overline{a'} b} \right) \nonumber\\
		&=\delta_{\tail(\overline{a})\tail(\overline{a'})}\sum_{b\in\arc_{\tail(\overline{a})*}}
				\left( |\tau(b)|^2 - \tau(b)\delta_{\overline{a'} b}
			- \overline{\tau(b)}\delta_{\overline{a} b} + \delta_{\overline{a} b} \delta_{\overline{a'} b} \right). \nonumber
	\end{align}
	For convenience, we show the $(\overline{a},\overline{a'})$-element of $M_{\rm S}M_{\rm S}^*$ below:
	\begin{align*}
		&\delta_{\tail(a)\tail(a')}\sum_{b\in\arc_{\tail(a)*}}
				\left( |\tau(b)|^2 - \tau(b)\delta_{a' b}
			- \overline{\tau(b)}\delta_{a b} + \delta_{a b} \delta_{a' b} \right) \nonumber \\
		&=\begin{cases}
			\left( \sum_{b\in\arc_{\tail(a)*}}|\tau(b)|^2 \right) - 2|\tau(a)|\cos(\arg \tau(a))+1 & {\rm if } \ \ a=a',\\
			\left(\sum_{b\in\arc_{\tail(a)*}}
				|\tau(b)|^2\right)  - \tau(a')
			- \overline{\tau(a)} & {\rm if} \ \ a\neq a' \ {\rm and} \ \tail(a)=\tail(a'), \\
			0 & {\rm otherwise}.
		\end{cases}
	\end{align*}

	We also assume that $(M_{\rm S}M_{\rm S}^*)_{a,a'}=\delta_{aa'}$.
	Comparing the $(a,a)$-element of $M_{\rm S}^*M_{\rm S}$ and the $(\overline{a},\overline{a})$-element of $M_{\rm S}M_{\rm S}^*$ gives
	\begin{align*}
		\deg(\tail(a))|\tau(a)|^2 = \sum_{b\in\arc_{\tail(a)*}}|\tau(b)|^2
	\end{align*}
	for $\forall a\in\arc$.
	Thus, $|\tau(a)| = |\tau(a')|$ holds for $\forall u\in V$ and $\forall a,a'\in\arc_{u*}$.
	Let $\varphi_a$ denote $\arg(\tau(a))$.
	For $a,a'\in\arc_{u*}$ with $a\neq a'$, from $(a,a')$-element of $M_{\rm S}^*M_{\rm S}$ and the $(\overline{a},\overline{a'})$-element of $M_{\rm S}M_{\rm S}^*$,
	we obtain
	\begin{align*}
		\deg(u)|\tau(a)|^2 &= \sum_{b\in\arc_{u*}}|\tau(b)|^2 \\
		&= \deg(u)\overline{\tau(a)}\tau(a') \\
		&= \deg(u)|\tau(a)|^2e^{i(\varphi_{a'}-\varphi_a)}.
	\end{align*}
	Thus, $e^{i(\varphi_{a'}-\varphi_a)}=1$ holds. 
	It follows that $\tau(a)=\tau(a')$ for $\forall u\in V$ and $\forall a\in\arc_{a,a'\in\arc_{u*}}$.
\end{proof}

\subsection{The quantum walks following from the generalized weighted zeta function}

We also show the unitary conditions for the edge matrix $M$ of the generalized weighted zeta function.
\begin{thm}\label{thm:main2}
	The edge matrix $M$ is unitary if and only if the maps $\tau$ and $\upsilon$ satisfy the following conditions: for each $u\in V$,
	\begin{itemize}
		\item If $\deg(u)=1$, then $|\tau(a) - \upsilon(a)|^2 = 1$ holds for $a\in \arc_{v*}$.
		\item If $\deg(u)\geq 2$, then 
			for $\forall a\in\arc_{u*}$ and $R_u \in \left[ -\frac{2}{d}, \ \frac{2}{d} \right]$, 
				\begin{align*}
				&|\upsilon(a)|=1, \\
				&\tau(a)=\upsilon(a)\left( \frac{\deg(\tail(a))R_u^2}{2} +i\frac{R_u\sqrt{4-\deg(\tail(a))^2R_u^2}}{2} \right).
				\end{align*}
	\end{itemize}
\end{thm}

\begin{proof}
	Let $\arc:=\arc(G)$ and assume that $M$ is unitary.
	The $(a, a')$-element of the complex conjugate transpose $M^*$ is given by
	\begin{align*}
		 \overline{\tau(a)}\delta_{\head(a')\tail(a)} - \overline{\upsilon(a)} \delta_{\overline{a'} a}.
	\end{align*}

	The $(a,a')$-element of $M^*M$ is as follows:
	\begin{align*}
		& \sum_{b\in\arc}
				\left( \overline{\tau(a)}\delta_{\head(b)\tail(a)} - \overline{\upsilon(a)}\delta_{\overline{a} b} \right)
				\left( \tau(a')\delta_{\head(b)\tail(a')} - \upsilon(a') \delta_{\overline{b} a'} \right) \nonumber\\
		& = \delta_{\tail(a)\tail(a')} \sum_{b \in\arc_{*\tail(a)}}
				\left( \overline{\tau(a)} - \overline{\upsilon(a)}\delta_{\overline{a} b} \right)
				\left( \tau(a') - \upsilon(a')\delta_{\overline{b} a'} \right) \nonumber\\
		& = \delta_{\tail(a)\tail(a')} \sum_{b \in\arc_{*\tail(a)}}
				\left( \overline{\tau(a)}\tau(a') - \overline{\tau(a)}\upsilon(a') \delta_{\overline{b} a'} -\tau(a')\overline{\upsilon(a)}\delta_{\overline{a} b} + \upsilon(a')\overline{\upsilon(a)}\delta_{\overline{b} a'}\delta_{\overline{a} b} \right) \nonumber\\
		& = \begin{cases}
			\deg(\tail(a)) |\tau(a)|^2 -\overline{\tau(a)}\upsilon(a)  -\tau(a)\overline{\upsilon(a)} + |\upsilon(a)|^2, & {\rm if}\quad a=a',\\
			\deg(\tail(a)) \overline{\tau(a)}\tau(a') -\overline{\tau(a)}\upsilon(a')  -\tau(a')\overline{\upsilon(a)} &  {\rm if}\quad  a\neq a' \ {\rm and} \  \tail(a)=\tail(a'), \\
			0 & {\rm otherwise}. 
		\end{cases}
	\end{align*}

	The $(a,a')$-element of $MM^*$ is as follows:
	\begin{align}
		& \sum_{b\in\arc}
		\left( \tau(b)\delta_{\head(a)\tail(b)} - \upsilon(b)\delta_{\overline{a} b} \right)
		\left( \overline{\tau(b)}\delta_{\head(a')\tail(b)} - \overline{\upsilon(b)}\delta_{\overline{a'} b} \right) \nonumber\\
		&= \delta_{\head(a)\head(a')} \sum_{b\in\arc_{\head(a)*}}
				\left( \tau(b) - \upsilon(b)\delta_{\overline{a} b} \right)
			\left( \overline{\tau(b)} - \overline{\upsilon(b)}\delta_{\overline{a'} b} \right) \nonumber\\
		&=\delta_{\head(a)\head(a')}\sum_{b\in\arc_{\head(a)*}}
				\left( |\tau(b)|^2 - \tau(b)\overline{\upsilon(b)}\delta_{\overline{a'} b}
			- \overline{\tau(b)}\upsilon(b)\delta_{\overline{a} b} + \upsilon(b)\overline{\upsilon(b)}\delta_{\overline{a} b} \delta_{\overline{a'} b} \right) \nonumber\\
		&=\delta_{\tail(\overline{a})\tail(\overline{a'})}\sum_{b\in\arc_{\tail(\overline{a})*}}
				\left( |\tau(b)|^2 - \tau(b)\overline{\upsilon(b)}\delta_{\overline{a'} b}
			- \overline{\tau(b)}\upsilon(b)\delta_{\overline{a} b} + \upsilon(b)\overline{\upsilon(b)}\delta_{\overline{a} b} \delta_{\overline{a'} b} \right). \nonumber
	\end{align}
	For convenience, we show the $(\overline{a},\overline{a'})$-element of $MM^*$ below:
	\begin{align}
		&\delta_{\tail(a)\tail(a')}\sum_{b\in\arc_{\tail(a)*}}
				\left( |\tau(b)|^2 - \tau(b)\overline{\upsilon(b)}\delta_{a' b}
			- \overline{\tau(b)}\upsilon(b)\delta_{a b} + \upsilon(b)\overline{\upsilon(b)}\delta_{a b} \delta_{a' b} \right) \nonumber\\
		&=\begin{cases}
		\left( \sum_{b\in\arc_{\tail(a)*}}|\tau(b)|^2 \right) - \overline{\tau(a)}\upsilon(a) -\tau(a)\overline{\upsilon(a)} +|\upsilon(a)|^2 & {\rm if } \ \ a=a', \\
			\left(\sum_{b\in\arc_{\tail(a)*}} |\tau(b)|^2\right)  - \tau(a')\overline{\upsilon(a')} - \overline{\tau(a)}\upsilon(a) & {\rm if} \ \ a\neq a' \ {\rm and} \ \tail(a)=\tail(a'), \\
			0 & {\rm otherwise}.
		\end{cases}\label{eq:genMM*}
	\end{align}
	
	Let us assume that $(M^*M)_{a,a'}=\delta_{a,a'}$ and $(MM^*)_{a,a'}=\delta_{a,a'}$.
	Since $(M^*M)_{a,a'}=(MM^*)_{\overline{a},\overline{a}}=0$ for $a,a'\in\arc$ with $\tail(a)\neq \tail(a')$,
	it is sufficient to consider $(M^*M)_{a,a'}$ and $(MM^*)_{\overline{a},\overline{a}}$ for $u\in V$ and $a,a'\in\arc_{u*}$.
	Let $d:=\deg(u)$.
	
	Suppose that $\arc_{u*}= \{ a \}$.
	The elements of $MM^*$ and $M^*M$ related to $a\in\arc_{u*}$ are $(MM^*)_{a,a}$ and $(MM^*)_{\overline{a},\overline{a}}$,
	and these are equal to each other and given by
	\begin{align*}
		|\tau(a)-\upsilon(a)|^2 =1.
	\end{align*}

	Suppose that $|\arc_{u*}| \geq 2$
	Comparing $(M^*M)_{a,a}$ and $(MM^*)_{\overline{a},\overline{a}}$ gives
	\begin{align*}
		d |\tau(a)|^2 = \sum_{b\in\arc_{u*}}|\tau(b)|^2.
	\end{align*}
	Since the above equation holds for $\forall a\in\arc_{u*}$, we get $|\tau(a)| = |\tau(a')|$ for $\forall a,a'\in\arc_{u*}$.
	Let $R_u:=|\tau(a)|$ for $\forall a\in\arc_{u*}$.
	If $R_u=0$, then 
	$$
		(M^*M)_{a,a'}=(MM^*)_{\overline{a},\overline{a'}}
		=\begin{cases}
			|\upsilon(a)|^2 & {\rm if } \quad a=a', \\
			0 & {\rm otherwise}.
		\end{cases}
	$$
	Thus, $|\upsilon(a)|=1$ holds for $\forall a\in\arc_{u*}$.

	We assume that $R_u\neq 0$.
	The imaginary part and real part of $(\overline{a},\overline{a'})$-element of (\ref{eq:genMM*}) are as follows:
	\begin{align*}
		\im (MM^*)_{\overline{a},\overline{a'}}
		&= \im \left(-\tau(a')\overline{\upsilon(a')} - \overline{\tau(a)}\upsilon(a) \right)	\\
		&= - R_{\tau} \left( |\overline{\upsilon(a')}|\sin(\arg \tau(a')\overline{\upsilon(a')} ) + |\upsilon(a)|\sin(\arg \overline{\tau(a)}\upsilon(a)) \right) \\
		&= - R_{\tau} \left( |\upsilon(a')|\sin(\arg \tau(a')\overline{\upsilon(a')} ) - |\upsilon(a)|\sin(\arg \tau(a)\overline{\upsilon(a)}) \right), \\
		\re(MM^*)_{\overline{a},\overline{a'}}
		&=dR_u^2  + \re \left(-\tau(a')\overline{\upsilon(a')} - \overline{\tau(a)}\upsilon(a) \right)	\\
		&= dR_u^2  -R_u \left( |\overline{\upsilon(a')}|\cos(\arg \tau(a')\overline{\upsilon(a')}) + |\upsilon(a)|\cos(\arg \overline{\tau(a)}\upsilon(a)) \right) \\
		&= dR_u^2  -R_u \left( |\upsilon(a')|\cos(\arg \tau(a')\overline{\upsilon(a')}) + |\upsilon(a)|\cos(\arg \tau(a)\overline{\upsilon(a)}) \right).
	\end{align*}
	For these to be equal $0$, the following must holds:
	\begin{align*}
		&|\upsilon(a')|\sin(\arg \tau(a')\overline{\upsilon(a')} ) 
		= |\upsilon(a)|\sin(\arg \tau(a)\overline{\upsilon(a)}), \nonumber\\
		&|\upsilon(a')|\cos(\arg \tau(a')\overline{\upsilon(a')})=|\upsilon(a)|\cos(\arg \tau(a)\overline{\upsilon(a)}) = \frac{dR_u}{2}.\label{eq:cos}
	\end{align*}
	That is, for $\forall a,a'\in\arc_{u*}$,
	\begin{align*}
		\tau(a)\overline{\upsilon(a)}&=\tau(a')\overline{\upsilon(a')} \\
		& = R_u \left(\frac{dR_u}{2} +i\sqrt{1-\left(\frac{dR_u}{2}\right)^2}\right) \\
		& =\frac{dR_u^2}{2} +i\frac{R_u\sqrt{4-d^2R_u^2}}{2},
	\end{align*}
	where $R_u\in \left[ -\frac{2}{d}, \ 0 \right) \cup \left( 0, \ \frac{2}{d} \right]$.
	The $(a,a)$-element of $M^*M$ is rewritten by
	\begin{align*}
		1 &= d R_u^2 - 2|\upsilon(a')|\cos(\arg \tau(a')\overline{\upsilon(a')}) +|\upsilon(a)|^2 \\
		 &= d R_u^2 -dR_u +|\upsilon(a)|^2 \\
		 &= |\upsilon(a)|^2,
	\end{align*}
	and we obtain $|\upsilon(a)|=1$ for $\forall a\in\arc_{u*}$.
	Thus, the following holds:
	\begin{align*}
		 \tau(a) 
		 = \upsilon(a) \left(\frac{dR_u^2}{2} +i\frac{R_u\sqrt{4-d^2R_u^2}}{2}\right) 
	\end{align*}
	holds for $\forall a\in\arc_{u*}$. 
	Note that the above also holds for $R_u=0$.

	Substituting $\tau(a')= \tau(a)\overline{\upsilon(a)}\upsilon(a')$ into the $(a,a')$-element of $M^*M$ gives
	\begin{align*}
		&d \overline{\tau(a)}\tau(a)\overline{\upsilon(a)}\upsilon(a') -\overline{\tau(a)}\upsilon(a')  -\tau(a)\overline{\upsilon(a)}\upsilon(a')\overline{\upsilon(a)} \\
		&= \overline{\upsilon(a)}\upsilon(a') 
			\left( d R_u^2 -\overline{\tau(a)}\upsilon(a) -\tau(a)\overline{\upsilon(a)} \right) \\
		&= R_u\overline{\upsilon(a)}\upsilon(a') 
			\left( d R_u  - 2 \cos (\arg \tau(a)\overline{\upsilon(a)}) \right) \\
		&= R_u\overline{\upsilon(a)}\upsilon(a') 
			\left( d R_u  - 2 \frac{dR_u}{2}\right) \\
		&=0.
	\end{align*}
	We see that $(M^*M)_{a,a'} = 0$ holds.

\end{proof}

\section*{Acknowledgements}
I would like to give heartfelt thanks to Professor Norio Konno and Professor Hideaki Morita
	who provided carefully considered feedback and valuable comments.
This work was supported by Grant-in-Aid for JSPS Fellows (Grant No. 20J20590).



\noindent
\begin{thebibliography}{000}


\bibitem{bartholdi99}
	L. Bartholdi, 
	Counting paths in graphs, 
	{\it Eiseign. Math.} {\bf 45} (1999), 83-131.

\bibitem{grover1996fast}
Lov K. Grover, 
A fast quantum mechanical algorithm for database search,
{\it Proceedings of the Twenty-Eighth Annual ACM Symposium on Theory of Computing} (1996)
212-219.

\bibitem{Higuchi2007}
Y. Higuchi, N. Konno, I. Sato, and E. Segawa,
Periodicity of the Discrete-time Quantum Walk on a Finite Graph,
{\it Interdisciplinary Information Sciences} {\bf 23} (2017), 
75-86


\bibitem{IIMSS21}
	Y. Ide, A. Ishikawa, H. Morita, I. Sato and E. Segawa, 
	The Ihara expression for the generalized weighted zeta function of a finite simple graph, 
	{\it Lin. Alg. Appl.} {\bf 627} (2021), 227-241.



\bibitem{ihara66}
	Y. Ihara, 
	On discrete subgroup of the two by two projective linear group over $p$-adic fields, 
	{\it J. Math. Soc. Japan} {\bf 18} (1966), 219-235.



\bibitem{konnosato12}
	N. Konno and I, Sato, 
	On the relation between quantum walks and zeta functions, 
	{\it Quantum Inf. Process.} {\bf 11} (2012), 341-349.



\bibitem{konno2017}
N. Konno, I. Sato, and E. Segawa,
The spectra of the unitary matrix of a 2-tessellable staggered quantum walk on a graph,
{\it Yokohama Math. J.} {\bf 62} (2017),
52-87.


\bibitem{morita20}
	H. Morita, 
	Ruelle zeta functions for finite digraphs, 
	{\it Lin. Alg. Appl.} {\bf 603} (2020), 329-358. 



\bibitem{sato07}
	I. Sato, 
	A new Bartholdi zeta function of a graph, 
	{\it Int. J. Algebra} {\bf 1} (2007), 269-281.


\bibitem{segawa2011}
E. Segawa,
Localization of quantum walks induced by recurrence properties of random walks,
{\it Journal of Computational and Theoretical Nanoscience} {\bf 10} (2013), 1583-1590.

 


\end{thebibliography}
\end{document}